\documentclass[11pt,twoside]{amsart}
\usepackage[all]{xy}   
        \usepackage {amssymb,latexsym,amsthm,amsmath}
        \usepackage{enumitem}
        
\long\def\symbolfootnote[#1]#2{\begingroup%
\def\thefootnote{\fnsymbol{footnote}}\footnote[#1]{#2}\endgroup}

\newcommand{\n}{\mathfrak n}
\newcommand{\x}{\mathfrak X}
\newcommand{\h}{\mathfrak h}

\newcommand{\diag}{\textup{diag}}

\makeatletter
\def\imod#1{\allowbreak\mkern10mu({\operator@font mod}\,\,#1)}
\makeatother

\newtheorem{theorem}{Theorem}[section]
\newtheorem{lemma}[theorem]{Lemma}
\newtheorem{corollary}[theorem]{Corollary}
\newtheorem{proposition}[theorem]{Proposition}

\newtheorem*{theorem*}{Theorem}
\theoremstyle{definition}

\newtheorem{example}[theorem]{Example}
\numberwithin{equation}{section}

\newcommand{\ignore}[1]{}

\newcommand{\mynote}[1]{}
\begin{document}
\setcounter{section}{0}
\title{Computing $n^{\rm th}$ roots in ${\rm SL}_2$ and Fibonacci polynomials}
\author{Amit Kulshrestha}
\address{IISER Mohali, Knowledge City, Sector 81, Mohali 140 306 INDIA}
\email{amitk@iisermohali.ac.in}
\author{Anupam Singh}
\address{IISER Pune, Dr. Homi Bhabha Road, Pashan, Pune 411 008 INDIA}
\email{anupamk18@gmail.com}
\thanks{This work is supported by DST, India through Indo-Russian research grant INT/RUS/RFBR/P-288. The first named author also acknowledges SERB grant EMR/2016/001516 for supporting this work.}
\today


\begin{abstract}
Let $k$ be a field of characteristic $\neq 2$. In this paper we study squares, cubes and their products in split and anisotropic groups of type $A_1$.
In split case, we show that computing $n^{\rm th}$ roots is equivalent to finding solutions of certain polynomial equations in at most two variables over the base field $k$. The description of these polynomials involves generalised Fibonacci polynomials.  Using this we obtain asymptotic proportions of
$n^{\rm th}$ powers, and conjugacy classes which are $n^{\rm th}$ powers, in ${\rm SL}_2(\mathbb F_q)$ when $n$ is a prime or $n = 4$.

We also extend already known Waring type result for ${\rm SL}_2(\mathbb F_q)$, that every element of ${\rm SL}_2(\mathbb F_q)$ is a product of two squares, to ${\rm SL}_2(k)$ for an arbitrary $k$. For anisotropic groups of type $A_1$, namely ${\rm SL}_1(Q)$ where $Q$ is a quaternion division algebra, we prove that when $2$ is a square in $k$, every element of ${\rm SL}_1(Q)$ is a product of two squares if and only if $-1$ is a square in ${\rm SL}_1(Q)$.
\end{abstract}

\maketitle

\section{Introduction}
Let $k$ be a field of characteristic $\neq 2$, and let $G$ be a linear algebraic group over $k$ which is a form of ${\rm SL}_2$. In this paper, we determine which elements $g\in G(k)$ admit solutions in $G(k)$ of equations of the type $X^n=g$ and $X_1^nX_2^n=g$. The split and anisotropic forms of ${\rm SL}_2$ over $k$ give rise to the groups ${\rm SL}_2(k)$ and ${\rm SL}_1(Q)$ where $Q$ is a quaternion central division algebra over $k$. When $X^n=g$ admits a solution for a given $g$, we explicitly determine all of them.  We reduce the problem of finding solutions of $X^n=g$ in ${\rm SL}_2(k)$ to finding simultaneous solutions of some polynomial equations in at most two variables. The equation $X^n=g$ for $g\in {\rm SL}_1(Q)$ can be solved in ${\rm SL}_1(Q)$ via embedding 
of ${\rm SL}_1(Q)$ in ${\rm SL}_2(K)$ where $K$ is a maximal subfield in $Q$. 

The proportion of $n^{\rm th}$ powers in a finite group, and an asymptotic formula for the same is of interest in combinatorics. See, for example~\cite{bl,po}, where the authors have computed this for the symmetric groups. We count the number $c(n,q)$ of conjugacy classes in ${\rm SL}_2(\mathbb F_q)$ which are $n^{\rm th}$ powers, and $s(n,q)$, the number of elements in ${\rm SL}_2(\mathbb F_q)$ which are $n^{\rm th}$ powers. This is the content of Theorem~\ref{count-prime-powers}. We compute these directly without requiring character theory. As an application to our counting an alternate proof of some well known Waring type theorems (see~\cite{lst1, lst2, lost, glost} for more general results) immediately follows in Corollary~\ref{two-n-powers}. 

Since our computation of determining $n^{\rm th}$ roots works over an arbitrary field, as an application, we prove some Waring type of results for forms of ${\rm SL}_2$. When characteristic ${\rm char}(k) \neq 2$, we prove that every element of split group ${\rm SL}_2(k)$ is a product of two squares. In addition, if $2$ is a square in $k$, we prove in the anisotropic case that every element of ${\rm SL}_1(Q)$ is a product of two squares if and only if $-1$ is a square in ${\rm SL}_1(Q)$. This is proved in Theorem~\ref{waring-type-main-result}. Additionally, in Theorem~~\ref{waring-type-main-result}
we also discuss results concerning products of higher powers in ${\rm SL}_2(k)$ and ${\rm SL}_1(Q)$.

{\bf Acknowledgment:} The authors would like to thank B. Sury, Indian Statistical Institute Bangalore, for his encouragement during this work. We thank the referees for helpful comments which improved the readability of this article.

\section{The groups of type $A_1$}
Let $k$ be a field of characteristic $\neq 2$. In this section, we set notation to be followed later. The groups of type $A_1$ over $k$ are forms of the algebraic group ${\rm SL}_2$ over $k$. The $k$-forms of ${\rm SL}_2$ are in one-one correspondence with the quaternion algebras over $k$. In general, the $k$ points of ${\rm SL}_2$ is given by ${\rm SL}_{Q}$ (the set of reduced norm $1$ elements of $Q$) where $Q$ is a quaternion algebra over $k$. The $k$ points of split form corresponds to the matrix algebra $M_2(k)$ and is denoted as ${\rm SL}_{2,k}$ and that of anisotropic form corresponds to a quaternion division algebra $Q$ denoted as ${\rm SL}_{1,Q}$. To set the notation for what follows we describe each case briefly. 

\subsection{The split form ${\rm SL}_{2,k}$}\label{split-intro}
First, we introduce the split form of the algebraic group ${\rm SL}_2$ over $k$. In what follows we simply write $G, B, T$ to be the $k$-rational points of ${\rm SL}_2$ (simply denoted as ${\rm SL}_2(k)$), the upper triangular Borel and the diagonal maximal torus respectively for the convenience of notation. We fix these notation following the theory of Chevalley groups as in~\cite{ca}. The set $B=\left\{\begin{pmatrix} a&b\\ & a^{-1} \end{pmatrix}\mid a\in k^*, b\in k\right\}$ consisting of upper triangular matrices is said to be the standard Borel subgroup of ${\rm SL}_2(k)$. The set of all diagonals is a maximal torus and is denoted as $T$. We denote the diagonal matrices as $\h(a)=\diag(a, a^{-1})$ for $a\in k^*$, and the root generators as $\x_{12}(t)=\begin{pmatrix} 1&t\\ &1  \end{pmatrix}$ and $\x_{21}(t)= \begin{pmatrix} 1& \\ t&1  \end{pmatrix}$ for $t\in k$. Then (see Lemma 6.1.1 in~\cite{ca}), the group ${\rm SL}_2(k)$ is generated by the set of all root generators $\{\x_{12}(t_1), \x_{21}(t_2) \mid t_1, t_2 \in k\}$. Define, $\n(\alpha):=\x_{12}(\alpha)\x_{21}(-\alpha^{-1})\x_{12}(\alpha) = \begin{pmatrix} & \alpha \\ -\alpha^{-1} & \end{pmatrix}$ for $\alpha\neq 0$ and note that $\h(a) = \n(a)\n(-1)$. To simplify the notation we denote the Weyl group element $\n(1)=\begin{pmatrix} & 1 \\ -1 &\end{pmatrix}$ simply by $\
n$. The Bruhat decomposition is the double coset decomposition of the group ${\rm SL}_2(k)$ with respect to the subgroup $B$, is ${\rm SL}_2(k)=B\bigsqcup B\n B$. The double cosets have group structure induced by the Weyl group. In this case, the Weyl group is simply $\mathbb Z/2\mathbb Z$. Every element of $B$ can be written uniquely as $\h(a)\x_{12}(s)$, where $a\in k^*, s\in k$. And, every element of $B\n B$ has unique expression as $\x_{12}(t) \n(a)\x_{12}(s)$ for some $t,s\in k$ and $a\in k^*$. For the computations later, we need several commuting relations among the elements we have defined earlier. We list them here. 
\begin{proposition}\label{commuting-formulae}
With the notation as above,
\begin{enumerate}
\item $\h(ab)=\h(a)\h(b)$, $\x_{12}(t_1+t_2)=\x_{12}(t_1)+\x_{12}(t_2)$, $\x_{21}(t_1+t_2)=\x_{21}(t_1)+\x_{21}(t_2)$ and $\n(\alpha) \n(\beta)= -\h(\alpha\beta^{-1})$.
\item $\h(a) \x_{12}(t)= \x_{12}(a^2t) \h(a)$,  $\x_{12}(t)\h(a) = \h(a)\x_{12}(a^{-2}t)$, $\h(a) \x_{21}(t)= \x_{21}(a^{-2}t) \h(a)$, $\h(a) \n(\alpha)= \n(a\alpha) = \n(a^2\alpha) \h(a)$.
\item $\n(\alpha) \x_{12}(t)= \x_{21}(-\alpha^{-2}t) \n(\alpha)$.
\item $\n(a) \x_{12}(t) \n(a) = \x_{12}(-a^{2}t^{-1}) \n(-a^{2}t^{-1}) \x_{12}(-a^{2}t^{-1})$.
\end{enumerate}
\end{proposition}
\noindent We will make use of these results freely as and when required. For the convenience of reader we note down multiplication relations as well.
\begin{proposition}\label{coset-mult}
With the notation as above,
\begin{enumerate}
\item $\h(\alpha_1)\x_{12}(\psi_1). \h(\alpha_2)\x_{12}(\psi_2) = \h(\alpha_1\alpha_2)\x_{12}(\alpha_2^{-2}\psi_1+\psi_2)$.
\item $\h(\alpha_1)\x_{12}(\psi_1).\x_{12}(\tau_2)\n(\alpha_2)\x_{12}(\psi_2) = \x_{12}\left(\alpha_1^2(\psi_1+\tau_1)\right) \n(\alpha_1\alpha_2) \x_{12}(\psi_2)$.
\item  $\x_{12}(\tau_1)\n(\alpha_1)\x_{12}(\psi_1). \x_{12}(\tau_2)\n(\alpha_2)\x_{12}(\psi_2)\\  =  \begin{cases}\x_{12}\left( \tau_1 - \frac{\alpha_1^2}{\psi_1+\tau_2}\right) \n\left( -\frac{\alpha_1\alpha_2}{\psi_1+\tau_2}\right)\x_{12}\left(\psi_2- \frac{\alpha_2^2}{\psi_1+\tau_2} \right) \ \ when\ \psi_1+\tau_2\neq 0, \\
\h(-\frac{a_1}{a_2}) \x_{12}(\frac{a_2^2}{a_1^2}\tau_1 + \psi_2) \ \ when\ \psi_1+\tau_2 = 0.                                                                                                    \end{cases}$
\end{enumerate}
\end{proposition}
\noindent We also follow the convention that the scalars which are used for $\n()$ and $\h()$ (such as $a, \alpha$) are invertible in $k$ and the elements which are used for $\x_{12}()$, such as $s,t, \tau, \psi$ are in $k$.

\subsection{The anisotropic form ${\rm SL}_{1,Q}$}\label{anisotropic-intro}
In this paper we reserve the notation $Q$ to denote a quaternion division algebra. Once again, for convenience of notation, we denote the $k$ points of anisotropic form of ${\rm SL}_2$, simply by ${\rm SL}_1(Q)$ which is the set of norm $1$ elements of $Q$.  
We describe the doubling construction here (see~\cite[\S 1.5]{sv}) and follow this notation in the following sections. 

Let $K$ be a degree $2$ field extension of $k$. We write $K=k(\zeta)$ where $\zeta^2=a\in k^*$. 
Let $x \to \bar{x}$ be the non-trivial $k$-automorphism of $K$ induced by $\zeta \mapsto -\zeta$.
Let $\lambda \in k^*$ be such that $\lambda \not\in N(K^*)$. Here $N : K^* \to k^*$ denotes the norm map for quadratic extension given by 
$N(x) = x\bar{x}$. Then, the quaternion algebra $Q=K\oplus K$ with multiplication defined by
$$(x_1,y_1).(x_2,y_2) = (x_1x_2+\lambda y_1\bar y_2, x_1y_2+\bar x_2y_1)$$
is a division algebra. This quaternion algebra is denoted by $Q = \left(\frac{a,\lambda}{k}\right)$. The conjugation involution in $Q$ is given by $\overline{(x,y)}=\bar x - y$ and norm is given by the formula $N((x,y))=N(x)-\lambda N(y)$ (see Proposition 1.5.1 of~\cite{sv}). We remark that, we use same notation for conjugation and norm in $Q$ as we do over the field $K$, while the purpose is clear from the context. The group 
$\{(x,y)\in Q : N((x,y))=1\}$ of norm $1$ elements in $Q$ is denoted by  ${\rm SL}_1(Q)$. We embed $Q$ into $M_2(K)$ as follows:
$$(x,y)\mapsto \begin{pmatrix} x & \lambda y \\ \bar y & \bar x \end{pmatrix}$$
and, the quaternion multiplication (trace and norm respectively) becomes matrix multiplication (trace and determinant respectively). We can further think of ${\rm SL}_1(Q) = ({\rm SL}_1(Q)\cap B) \bigsqcup ({\rm SL}_1(Q)\cap B\n B)$ where $B$ is a Borel in ${\rm SL}_2(K)$. This allows us to use computations in split ${\rm SL}_2$ to be used in this case as well. Notice that $\begin{pmatrix} x & \lambda y \\ \bar y & \bar x \end{pmatrix} = \x_{12}\left(\frac{x}{\bar y}\right)\n\left(\frac{-1}{\bar y}\right) \x_{12}\left(\frac{\bar x}{\bar y}\right)$ for $y\neq 0$.

\subsection{Conjugacy classes in ${\rm SL}_2(\mathbb F_q)$}
This can be found in many textbooks on representation theory of finite groups. The size of the group is $q(q^2-1)$. In all, there are $q+4$ conjugacy classes. We list them below according to their types. 

{\bf Central classes:} The two elements $\pm 1$, represented as $\h(\pm 1)$, are in the centre and form distinct conjugacy classes. 

{\bf Split regular semisimple classes: } These are the diagonal matrices represented by elements $\h(a)$ with $a\neq \pm 1$. There are $\frac{q-3}{2}$ such conjugacy classes and each conjugacy class is of size $q(q+1)$.

{\bf Non-semisimple classes:} There are $4$ such conjugacy classes given by $\x_{12}(1)$, $\h(-1)\x_{12}(-1)$, $\x_{12}(\epsilon)$, $\h(-1)\x_{12}(-\epsilon)$ where $\epsilon$ is a fixed non-square in $\mathbb F_q$. The size of each conjugacy class is $\frac{(q-1)(q+1)}{2}$.

{\bf Anisotropic regular semisimple classes:} These conjugacy classes are given by the companion matrix of an irreducible polynomial $X^2- \delta X+1$ of degree $2$. Thus, these classes are represented by the companion matrix $\n(-1) \x_{12}(\delta)$ where $\delta\in \mathbb F_q$ satisfying $\delta^2 - 4$ a non-square in the field. There are $\frac{q-1}{2}$ such conjugacy classes and each one of them are of size $q(q-1)$. In Section~\ref{sec-counting}, we present a different perspective to look at these classes which is useful in counting. 

\section{Generalised Fibonacci polynomials and $n^{\rm th}$ power}\label{GFP}
In this section, we define Fibonacci polynomials in two variables which appear in our study later.  Denote $u_0(X,Y)=0$, $u_1(X,Y)=1$ and define recursively, 
$$u_r(X,Y)=Xu_{r-1}(X,Y) + Yu_{r-2}(X,Y).$$ 
Thus, $u_2(X,Y)=X$, $u_{3}(X,Y)=X^2+Y$, $u_4(X,Y)=X^3+2XY$ and so on. These polynomials have been studied in literature, for example, see~\cite{hl} and~\cite{acms}. We mention a few interesting properties keeping in mind that these polynomials are in $\mathbb Z[X,Y]$.
\begin{proposition}
With the notation as above,
\begin{enumerate}
\item For $n\geq 2$, $u_n \mid u_m$ if and only if $n\mid m$.
\item $(u_n, u_m) = u_{(n,m)}$. This implies, $(u_n, u_{n+1})=1$.
\item The polynomials $u_n(X,Y)$ is irreducible over $\mathbb Q$ if and only if $n$ is a prime.
\item $u_{m+n+1}= u_{m+1}u_{n+1} + Yu_m u_n$.
\end{enumerate}
\end{proposition}
\noindent If we put $\phi(X,Y)= \frac{X+\sqrt{X^2+4Y}}{2}$ and $\psi(X,Y)= \frac{X-\sqrt{X^2+4Y}}{2}$ then $$u_n = \frac{\phi^n - \psi^n}{\phi - \psi}.$$

In our study, we come across the homogeneous version of these polynomials obtained as follows. Define $f_n(X,Y)=u_{n+1}(X,-Y^2)$. Therefore, $f_{-1}(X,Y)=0$, $f_0(X,Y)=1$ and recursively,  
$$f_r(X,Y)=Xf_{r-1}(X,Y) - Y^2f_{r-2}(X,Y).$$  
For example, $f_1=X$, $f_2=X^2-Y^2$, $f_3=X^3-2XY^2$, $f_4=X^4-3X^2Y^2+Y^4$, $f_5=X^5-4X^3Y^2+3XY^4$, $f_6= X^6-5X^4Y^2 + 6X^2Y^4-Y^6$ and so on. Clearly all these polynomials belong to $\mathbb Z[X,Y]$ and are homogeneous of degree $r$. In this paper, we refer to these polynomials $f_n(X,Y)$ as generalised Fibonacci polynomials. 

We compute a formula for $x^n$ when $x\in B\n B$. This formula is, interestingly, related to generalised Fibonacci polynomials $f_n(X,Y)$ defined above. We formulate this as a recursive relation in the generic case when the powers are not in $B$. 
\begin{proposition}\label{powerformula}
Let $x= \x_{12}(t) \n(a) \x_{12}(s)$ be in $B\n B$ and suppose $x^i$ belongs to $B\n B$ for all $1\leq i \leq n$. Let us denote $x^i=\x_{12}(t_i) \n(a_i) \x_{12}(s_i)$ where $t_1=t, a_1=a $ and $s_1=s$. Then, $x^n = \x_{12}(t_n) \n(a_n) \x_{12}(s_n)$ where,
$$t_n=t - \frac{a^2}{t + s_{n-1}}, s_n =s - \frac{a^2}{t + s_{n-1}}, a_n = -\frac{a_{n-1}a}{t + s_{n-1}}.$$
\end{proposition}
\begin{proof} We compute,
\begin{eqnarray*}
&& x^n= x^{n-1}.x = \x_{12}(t_{n-1}) \n(a_{n-1}) \x_{12}(s_{n-1}) . \x_{12}(t) \n(a) \x_{12}(s) \\ 
&=& \x_{12}(t_{n-1}).\n(a_{n-1}).\x_{12}(s_{n-1}+t) \n(a) \x_{12}(s)\\
&=& \x_{12}(t_{n-1}) \h\left(\frac{a_{n-1}}{a}\right) . \n(a) \x_{12}(s_{n-1}+ t) \n(a).\x_{12}(s)\\
&=& \x_{12}(t_{n-1})\h\left(\frac{a_{n-1}}{a}\right) . \x_{12}\left(-\frac{a^2}{t+s_{n-1}}\right) \n\left(-\frac{a^2}{t+ s_{n-1}} \right) \x_{12}\left(-\frac{a^2}{t+ s_{n-1}}\right). \x_{12}(s) \\
&=&  \x_{12}(t_{n-1})\h\left(\frac{a_{n-1}}{a}\right) . \x_{12}\left(-\frac{a^2}{t+ s_{n-1}}\right) \n\left( - \frac{a^2}{t+ s_{n-1}} \right) \x_{12}\left(s - \frac{a^2}{t + s_{n-1}}\right)\\
&=& \x_{12}(t_{n-1}) \x_{12}\left(-\frac{a_{n-1}^2a^2}{a^2(t + s_{n-1})}\right) \h\left(\frac{a_{n-1}}{a}\right) \n\left(- \frac{a^2}{t + s_{n-1}} \right) \x_{12}\left( s-\frac{a^2}{t + s_{n-1}}\right)\\
&=& \x_{12}\left(t_{n-1} - \frac{a_{n-1}^2}{t + s_{n-1}}\right)   \n\left(-\frac{a_{n-1}a}{t + s_{n-1}}\right ) \x_{12}\left(s-\frac{a^2}{t + s_{n-1}}\right).
\end{eqnarray*}
The first formula follows by symmetry of computation if we compute $x.x^{n-1}$.
\end{proof}
\noindent Now we rewrite these formulae involving generalised Fibonacci polynomials. Appearance of these polynomials in computing powers is well known (for example, see~\cite{ms}). To begin with, $t+s=f_1(t+s,a)$. Now, note that $t+s_2= t + \left(s - \frac{a^2}{t+s} \right) = \frac{(t+s)^2-a^2}{t+s}$. Thus we get, $(t+s)(t+s_2)= (t+s)^2-a^2=f_2(t+s,a)$. More generally, we have the following,
\begin{lemma} With the notation as above, for $r\geq 1$, 
$$\prod_{i=1}^{r}(t+s_i)=(t+s) \prod_{i=1}^{r-1}(t+s_i) - a^2 \prod_{i=1}^{r-2}(t+s_i)=f_r(t+s,a).$$
\end{lemma}
\begin{proof}
We already noted that $t+s=f_1(t+s,a)$ and $(t+s_1)(t+s_2)=f_2(t+s,a)$. Thus, if we prove the recursive relation we would have established the second identity. We note that, for $r\geq 3$,
\begin{eqnarray*}
(t+s) \prod_{i=1}^{r-1}(t+s_i) - a^2 \prod_{i=1}^{r-2}(t+s_i) &=& \left(\prod_{i=1}^{r-1}(t+s_i)\right) \left( (t+s) - \frac{a^2}{t+s_{r-1}}\right) \\
 &=&  \left(\prod_{i=1}^{r-1}(t+s_i) \right) (t + s_r) =  \prod_{i=1}^{r}(t+s_i).
\end{eqnarray*}
\end{proof}
\noindent Thus, the formulae in Proposition~\ref{powerformula} can be re-written as follows: 
\begin{eqnarray}\label{powerformulaf} 
t_n &=& t - a^2\frac{f_{n-2}(t+s,a)}{f_{n-1}(t+s,a)},  \\ 
s_n &= & s - a^2\frac{f_{n-2}(t+s,a)}{f_{n-1}(t+s,a)}, \nonumber \\  
a_n  &=& (-1)^{n-1}\frac{a^n}{ f_{n-1}(t+s,a)} \nonumber
\end{eqnarray}
where the last equation is obtained inductively as follows
$$a_n= -\frac{a_{n-1}a}{t + s_{n-1}} = (-1)^2 \frac{a_{n-2}a^2}{(t + s_{n-2})(t + s_{n-1}) } = \cdots = (-1)^{n-1}\frac{a^n}{f_{n-1}(t+s,a)}.$$

Now we deal with the case when $x\in B\n B$. It can happen that a certain power $x^r$ is in $B$. We determine this in the following,
\begin{proposition}\label{bnbpowerb}
Let $x=\x_{12}(t) \n(a) \x_{12}(s)\in B\n B$. Then, 
\begin{enumerate}
\item $f_1(t+s, a), \ldots, f_{r-2}(t+s, a)$ are all non-zero and $f_{r-1}(t+s, a)=0$ if and only if $x, \ldots, x^{r-1} \in B\n B$ and $x^r\in B$. 
\item Let $r$ be the smallest power such that $x^r\in B$ (i.e., $x^{r-1}\notin B$). Then, $x^n\in B$ if and only if $r\mid n$. 
\end{enumerate}
\end{proposition}
\begin{proof} Let us prove (1) first.
Clearly if $f_{1}(t+s,a)$ is not zero, we can compute $x^2$ by the formula~\ref{powerformulaf} and, hence $x^2$ belongs to $B\n B$. Thus if $f_1(t+s,a), \ldots, f_{i-1}(t+s, a)$ all are non-zero, we can compute $x^i$ by the formula and hence it belongs to $B\n B$ not in $B$. For converse, let us assume, $f_1(t+s, a), \ldots, f_{r-2}(t+s, a)$ are all non-zero and $f_{r-1}(t+s, a)=0$. 

Let us verify this for $r=2$ first, i.e., we have $f_1(t+s,a)=t+s=0$. In this case, $x=\x_{12}(t) \n(a) \x_{12}(-t) = \x_{12}(t) \n(a) \x_{12}(t)^{-1}$ and $x^2= \x_{12}(t) \n(a)^2 \x_{12}(t)^{-1} = -1$. Thus, $x^2\in B$ and $x^n=(-1)^m$ if $n=2m$ and $x^n=(-1)^mx$ if $n=2m+1$. Thus $x^n\in B$ if and only if $n$ is even.

Now to prove the general case, we note that $0=f_{r-1}(t+s,a)= (t+s)f_{r-2}(t+s,a) - a^2 f_{r-3}(t+s,a)$ (Fibonacci relation), gives, $(t+s)f_{r-2}(t+s,a) = a^2 f_{r-3}(t+s,a)$. Thus, $x^{r-1}$ can be computed by the formula~\ref{powerformulaf} and $t_{r-1}= t - a^2\frac{f_{r-3}(t+s,a)}{f_{r-2}(t+s,a)} = t-a^2\frac{(t+s)}{a^2} = -s$, similarly, $s_{r-1}=-t$. Thus, 
\begin{eqnarray}
x^r &=& x.x^{r-1} =  \x_{12}(t) \n(a) \x_{12}(s) . \x_{12}(t_{r-1}) \n(a_{r-1}) \x_{12}(s_{r-1})  \\
&=& \x_{12}(t) \n(a) \x_{12}(s). \x_{12}(-s) \n(a_{r-1}) \x_{12}(-t) \nonumber \\
&=& \x_{12}(t) \n(a) \n(a_{r-1}) \x_{12}(-t) = \x_{12}(t) \h(-aa_{r-1}^{-1}) \x_{12}(-t) \nonumber \\
&=& \h\left(-\frac{a}{a_{r-1}}\right) \x_{12}\left( \left(\frac{a_{r-1}^2}{a^2}-1\right)t\right) \nonumber
\end{eqnarray}
which belongs to $B$.

To prove (2), if $r \mid n$ then $x^n=(x^r)^{\frac{n}{r}} \in B$. Now suppose $x^n\in B$. Since $r$ is smallest such that $x^r\in B$ and $x^i\not\in B$ for all $i$, with $1\leq i \leq r-1$, i.e., $x^i\in B\n B$. Now write $n=lr+m$ where $0\leq m \leq r-1$ and then $x^n=(x^r)^lx^m$. Thus, $x^n\in B$ if and only if $m=0$, which happens if and only if $r\mid n$.
\end{proof}
\noindent In the next section we use the above to determine if an element of ${\rm SL}_2(k)$ has $n^{\rm th}$ root.

\section{$n^{\rm th}$ root in ${\rm SL}_2(k)$}\label{sec-nthroot}
Let $g$ be an element of ${\rm SL}_2(k)$ where ${\rm char}(k)\neq 2$. We want to solve the equation $X^n=g$ in ${\rm SL}_2(k)$. We make two separate cases depending on if $g$ is in $B$ or in $B\n B$.

\subsection{$n^{\rm th}$ roots in Borel}
Let $g= \h(\alpha) \x_{12}(\psi)\in B$ and let $x\in {\rm SL}_2(k)$ be a solution of $X^n=g$. We make two cases separately depending on if $x$ is in $B$ or $B\n B$. For an element $\alpha\in k$, define the following polynomials, $S_{2m}(\alpha, X) = (1+\alpha)(1+X^2+X^4+\cdots + X^{2(m-1)})$ and $S_{2m+1}(\alpha, X)=1+\alpha X + X^2+\alpha X^3 + X^4+\cdots + \alpha X^{2m-1}+ X^{2m}$. We remark that $S_{2m}(-1, X)=0$.
\begin{proposition}\label{n-th-powers-in-Borels}
Let $g= \h(\alpha) \x_{12}(\psi)\in B$. Then,  
\begin{enumerate}
\item for $\alpha=\pm 1$, the equation $x^n=g$ has a solution $x\in B$ if, and only if, the equations $X^n =\alpha$ and $S_{n}(\alpha, X)Y- \psi X^{2(n-1)}=0$ have simultaneous solution  for $X,Y$ in $k$.  
\item For $\alpha\neq \pm 1$, the equation $x^n=g$ has a solution $x\in B$ if, and only if, the equation $X^n =\alpha$ has a solution $X\in k$.
\end{enumerate}
\end{proposition}
\begin{proof}
Let $x = \h(a) \x_{12}(t) \in B$ be a root of $X^n = g$. 
Let us compute using formulae in Proposition~\ref{commuting-formulae}, $x^2= \h(a^2) \x_{12}((1+a^{-2})t)$, $x^3= \h(a^3) \x_{12}((1+a^{-2}+a^{-4})t)$ and inductively, $x^n= \h(a^n) \x_{12}((1+a^{-2}+a^{-4}+\cdots + a^{-2(n-1)})t)$. Thus $x^n=g$ gives two equations, 
$$a^n=\alpha, \hskip5mm (1+a^{-2}+a^{-4}+\cdots + a^{-2(n-1)})t=\psi.$$
Clearly to show that the solution $x$ exists we need to solve these two equations for $a$ and $t$. 

In case of (1), the equations are $a^n=1$ and $(1+a^{-2}+a^{-4}+\cdots + a^{-2(n-1)})t=\psi$. By multiplying the second equation with $a^{-2(n-1)}$ we get the required formula.

In case of (2), by multiplying with $a^{-2}$ to the second equation and subtracting with itself we get $(1-a^{-2n})t=(1-a^{-2})\psi$. Thus to get $a$ we need to solve the equation $X^n=\alpha$ over $k$ and to get $t$ we need to make sure $a^{-2n}\neq 1$, i.e., $\alpha^2\neq 1$. Conversely, the solution $x= \h(a) \x_{12}(t)$ is determined by the root $a^n=\alpha$ provided $a^{2n}\neq 1$ and $t=\left(\frac{1-a^{-2}}{1-a^{-2n}}\right)\psi$. 
\end{proof}

Now, we deal with the case if solution $x$ comes from $B\n B$. First, we deal with some small order cases.
\begin{proposition}\label{B-n-th-power-BnB}
Let $g= \h(\alpha) \x_{12}(\psi)\in B$. Then, 
\begin{enumerate}
\item the equation $x^2 =g$ has a solution $x$ in $B\n B$ if, and only if, $g=-1$. Further, the solutions are of the form $x=\x_{12}(t)\n(a)\x_{12}(-t)$.
\item The equation $x^3=g$ has a solution $x$ in $B\n B$ if, and only if, $g=\pm 1$. Further, the solutions are of the form $x=\x_{12}(a -s) \n(a)\x_{12}(s)$ or $\x_{12}(-a -s) \n(a)\x_{12}(s)$.
\item The equation $x^4=g$ has a solution $x$ in $B\n B$ if, and only if, $g=\pm 1$. The solutions for $g=1$ come from that of $x^2=-1$ in $B\n B$. The equation $x^4=-1$ has a solution $x$ in $B\n B$ if, and only if, $X^2-2Y^2=0$ has a solution $X,Y$ in $k$ with $Y\neq 0$. Further, the solutions are of the form $\x_{12}(t)\n(a)\x_{12}(-t+\gamma a)$ where $\gamma^2=2$.
\end{enumerate}
\end{proposition}
\begin{proof}
We begin with proving (1). Let $x=\x_{12}(t)\n(a)\x_{12}(s) \in B\n B$ be such that $x^2\in B$. For this it follows from Proposition~\ref{bnbpowerb} that $f_1(t+s,a)=t+s=0$. Hence $x=\x_{12}(t)\n(a)\x_{12}(-t)= \x_{12}(t)\n(a)\x_{12}(t)^{-1}$ and $x^2=-1$.

For (2), if $x^3\in B$ we must have $f_1(t+s,a)=t+s\neq 0$ and $f_2(t+s, a)=(t+s)^2-a^2=0$. And $x^3= \h\left(-\frac{a}{a_{2}}\right) \x_{12}\left( \left(\frac{a_{2}^2}{a^2}-1\right)t\right) = \h\left(\frac{t+s}{a}\right) \x_{12}\left( \left(\frac{a^2}{(t+s)^2}-1\right)t\right) = \h(\pm 1)=\pm 1$ since $a_2=-\frac{a^2}{t+s}$.

For (3), if $x^4\in B$, we make two cases, first when $x^2\in B$. From part (1), this happens when $x^2=-1$ and that would give $x^4=1$. These solutions are conjugates of $\n(a)$. The second case is when $x^2\not\in B$, thus we have $f_1(t+s, a)=t+s\neq 0$, $f_2(t+s,a)=(t+s)^2-a^2\neq 0$ and $f_3(t+s,a)= (t+s)^3-2(t+s)a^2=0$. The last equation gives that $(t+s)^2-2a^2=0$. Now using the formula $a_3=\frac{a^3}{(t+s)^2-a^2}=\frac{a^3}{2a^2-a^2}=a$, we compute $x^4 = \h\left(-\frac{a}{a_{3}}\right) \x_{12}\left( \left(\frac{a_{3}^2}{a^2}-1\right)t\right) = \h\left(-1\right) = -1$. Thus the solution exists only if $g=\pm 1$. We also note that, in the case $g=-1$, the solution exists if and only if $f_3=0$ which is equivalent to having solutions of $X^2-2Y^2=0$.
\end{proof}

\begin{proposition}\label{n-th-powers-in-BnB}
Let $g= \h(\alpha) \x_{12}(\psi)\in B$ and $n\geq 5$. 
\begin{enumerate}
\item When $\alpha\neq \pm 1$, the equation $x^n =g$ has a solution $x$ in $B\n B$ if, and only if, the equations $f_{r-3}(X,Y)^d - (-1)^{d(r-1)}\alpha X^dY^{d(r-4)}=0$ and $f_{r-1}(X,Y)=0$ have simultaneous solutions over $k$ with $Y$ non-zero, for some $d < n$ such that $dr= n$.
\item When $\alpha=\pm 1$, the equation $x^n=g$ has a solution $x$ in $B\n B$ if, and only if, $\psi=0$. In which case, the solutions are of the form $x=\x_{12}(t)\n(a)\x_{12}(\gamma-t)$ where $X=\gamma$ and $Y=a$ are solutions of the  equations $$f_{r-3}(X,Y)^d - (-1)^{d(r-1)} \alpha X^d Y^{d(r-4)}=0, \ \text{and}\ \ f_{r-1}(X,Y)=0$$ with $Y$ non-zero, for some $d < n$ such that $dr= n$.
\end{enumerate}
\end{proposition}
\begin{proof}
Let $x=\x_{12}(t)\n(a)\x_{12}(s)$ be a solution of $X^n=g$. Then $x^n=g\in B$. Thus from Proposition~\ref{bnbpowerb} there exists (smallest) $r$ such that $r\mid n$ and $x^r\in B$. Write $n=rd$. Now using the formula in the proof of same Proposition we have,
\begin{eqnarray*}
x^n&=&(x^r)^d = \left(\h\left(-\frac{a}{a_{r-1}}\right) \x_{12}\left( \left(\frac{a_{r-1}^2}{a^2}-1\right)t\right) \right)^d \\
&=& \h\left( \left(-\frac{a}{a_{r-1}}\right)^d\right) \x_{12}\left( (1+A+A^2+\cdots + A^{d-1})(A-1)t\right)\\
&=& \h\left( \left(-\frac{a}{a_{r-1}}\right)^d\right) \x_{12}\left((A^d-1)t\right)
\end{eqnarray*}
where $A= \left(\frac{a_{r-1}}{a}\right)^2$. Equating this with $g$ we get $ \left(-\frac{a}{a_{r-1}}\right)^d=\alpha$ and $(A^d-1)t=\psi$. That is, $\alpha^2=A^{-d}$ and hence $(\alpha^{-2}-1)t=\psi$. Now using the formula for $a_{r-1}$ in terms of Fibonacci polynomials and noting that $f_{r-1}=0$ we get $\left(-\frac{a}{a_{r-1}}\right) = (-1)^{r-1}\frac{f_{r-3}}{(t+s)a^{r-4}}$. Raising to the power $d$ we get the required equation. 

Now, to prove (1) we get the two equations as above. To prove the converse we need to determine $t$. But this is clear as $\alpha^2\neq 1$.

To prove (2), we note that $x^n=\h\left( \left(-\frac{a}{a_{r-1}}\right)^d\right)$ since $\alpha^2=1$. The rest of the proof as stated earlier.
\end{proof}

\subsection{$n^{\rm th}$ roots in $B\n B$}

Let $g= \x_{12}(\tau) \n(\alpha) \x_{12}(\psi)\in B\n B$ and let $x\in {\rm SL}_2(k)$ be a solution of $X^n=g$. Since $B$ is a subgroup the solution $x$ can not belong to $B$.  
We prove the following,
\begin{proposition}\label{BnB-n-th-root-BnB}
For $g=\x_{12}(\tau)\n(\alpha) \x_{12}(\psi)\in B\n B$, the equation $x^n=g$ has a solution $x$ in ${\rm SL}_2(k)$ if, and only if, the following equations have solution $X,Y$ in $k$:
\begin{enumerate}
 \item $2\alpha f_{n-2}(X,Y) + (-1)^{n-2}XY^{n-2} + (-1)^{n-1}(\tau +\psi) Y^{n-2} = 0$, and,
 \item $\alpha f_{n-1}(X,Y) + (-1)^{n}Y^n = 0$
\end{enumerate}
where $f_n(X,Y)$ denotes the generalised Fibonacci polynomials.
\end{proposition}
\begin{proof}
Let $x=\x_{12}(t)\n(a)\x_{12}(s)$ be a solution, i.e., $x^n=g$. Thus we get $t_n=\tau$, $s_n=\psi$ and $a_n=\alpha$. Using the formulae~\ref{powerformulaf} we get the following:
\begin{eqnarray}
\tau &=& t - a^2\frac{f_{n-2}(t+s,a)}{f_{n-1}(t+s,a)},  \\ 
\psi &= & s - a^2\frac{f_{n-2}(t+s,a)}{f_{n-1}(t+s,a)}, \nonumber \\  
\alpha  &=& (-1)^{n-1}\frac{a^n}{ f_{n-1}(t+s,a)} \nonumber
\end{eqnarray}
We add the first two equations and substitute $X=t+s$ and $Y=a$, to get,
$$\tau+\psi = X - 2Y^2 \frac{f_{n-2}(X,Y)}{f_{n-1}(X,Y)}.$$
The last equation becomes $\alpha f_{n-1}(X,Y) + (-1)^{n}Y^n = 0$ which is the second required equation in the theorem. 
Now we substitute this and get,
$$\tau+\psi = X - (-1)^{n-1} 2Y^2 \frac{f_{n-2}(X,Y) \alpha}{Y^n}$$
which is the required first equation. 

For converse, let $T$ and $a$ be a solution to the equations. That is we know, $t+s=T$ and $a$. We need to show existence of $t,s$ and $a$, so that $x^n=g$. Then the second equation gives, $f_{n-1}(t+s, a) = (-1)^{n-1}\frac{a^n}{\alpha}$. And the first equation gives $f_{n-2}(t+s,a)$ and hence we can determine $t$ and $s$ separately.
\end{proof}
In general it is difficult to separate out the variables $X$ and $Y$ from above equations. However, for $n=2,3$ and $4$ we can do better and reduce these equations to simpler equations. This we discuss in the following sections. Now we apply our results obtained so far to compute powers.

\section{Counting powers in ${\rm SL}_2(\mathbb F_q)$}\label{sec-counting}
Let $c(n,q)$ be the number of conjugacy classes in ${\rm SL}_2(\mathbb F_q)$ which are $n^{\rm th}$ power, and,  $s(n,q)$ be the number of elements in ${\rm SL}_2(\mathbb F_q)$ which are $n^{\rm th}$ power. In this section, we compute this number for ${\rm SL}_2(\mathbb F_q)$. Clearly, when $n \nmid (q^3-q) = |{\rm SL}_2(\mathbb F_q)|$ then $c(n,q)=q+4$ and $s(n,q)=(q^3-q)$. Thus, in what follows we assume $n\mid (q^3-q)$. We further compute the asymptotic formula for the ratio of conjugacy classes which are $n^{\rm th}$ powers, $c(n)=\displaystyle{\lim_{q\to \infty}}\frac{c(n,q)}{q+4}$ and the ratio of elements which are $n^{\rm th}$ powers,  
$s(n)=\displaystyle{\lim_{q\to \infty}}\frac{s(n,q)}{q^3-q}$. The main theorem is as follows:
\begin{theorem}\label{count-prime-powers}
Suppose $q$ is odd.  Then,
\begin{enumerate}
\item For $n=2$,
\[\begin{tabular}{|c|c|c|}\hline
 $q \imod 4$ & $c(2,q)$ & $s(2,q)$   \\ \hline
 $1$ & $\frac{q+5}{2}$ & $\frac{q^2(q-1)}{2} -q + 1$  \\ \hline 
 $3$ & $\frac{q+5}{2}$ & $\frac{q^2(q-1)}{2} + 1$   \\ \hline  
\end{tabular}
\]
Thus, $c(2)=s(2)=\frac{1}{2}$. 
\item For $n=4$, 
\[ 
\begin{tabular}{|c|c|c|}\hline
$q \imod 8$ & $c(4,q)$ & $s(4,q)$\\ \hline 
$1$ &  $\frac{3q+21}{8}$ & $\frac{3}{8}q^3 - \frac{1}{2}q^2 -\frac{7}{8}q +1$ \\ \hline
$3$ &  $\frac{3q+15}{8}$ & $\frac{3}{8}(q^3-q)$ \\ \hline
$5$ &  $\frac{3q+17}{8}$ & $\frac{3}{8}(q^3-q)$ \\ \hline 
$7$ &  $\frac{3q+11}{8}$ & $\frac{3}{8}q^3 -\frac{1}{2}q^2 + \frac{1}{8}q +1.$ \\ \hline
\end{tabular}
\]
Thus, $c(4)=s(4)=\frac{3}{8}$. 
\item When $n$ is an odd prime then,
\[
\begin{tabular}{|c|c|c|} \hline 
 & $c(n,q)$ & $s(n,q)$ \\ \hline
$n \mid q$ & $q$ & $(q-2)(q^2-1)$\\ \hline
$n \mid (q-1)$ & $\frac{(n+1)(q-1)}{2n} + 5$ &  $\frac{(n+1)(q^3-q)}{2n}$ \\ \hline
$n \mid (q+1) $& $\frac{(n+1)(q-3)+4}{2n} + 5$ & $ \frac{(n+1)(q^3-q)}{2n}.$ \\ \hline
\end{tabular}
\]
Thus, $c(n)=s(n)=\frac{n+1}{2n}$.
\end{enumerate}
\end{theorem}
\noindent Rest of the section is devoted to the proof of this theorem.

\begin{lemma}\label{squares-in-sl2}
Suppose ${\rm char}(k)\neq 2$. Then the set of squares in ${\rm SL}_2(k)$ is the union of following disjoint subsets.
\begin{enumerate}
\item $\mathcal S_1 = \{\h(a^2) \x_{12}(t) \mid a \in k^*\ with\ a^2 + 1 \neq 0,  t \in k\}$.
\item $\mathcal S_2 = \{-1\}  \cup \{\x_{12}(s) \n(b) \x_{12}\left(-s + b (a^2 - 2) \right) \mid a, b \in k^*, s \in k\}$.
\end{enumerate}
\end{lemma} 
\begin{proof}
The first set, except possibly the element $-1$, is obtained by squaring elements of $B$. For $x=\h(a)\x_{12}(t)\in B$ we note that $x^2= \h(a^2)\x_{12}\left((1+a^{-2})t\right) \in \mathcal S_1$ except when $a^2=-1$. The case when $a^2=-1$ gives the element $-1$ which is also obtained by squaring certain elements of $B\n B$, for example, $\n$. 
 
Now to get $\mathcal S_2$, we square elements of the set $B\n B$. Thus this gives elements of the form $\x_{12}(\tau)\n(\alpha)\x_{12}(\psi)$ such that $2-\frac{\tau+\psi}{\alpha}\in {k^*}^2$. This gives the required set.
\end{proof}

\begin{proof}[\bf{Proof of Theorem~\ref{count-prime-powers}(1)}]
We count the cardinalities of $\mathcal S_1$ and $\mathcal S_2$ as in Lemma~\ref{squares-in-sl2}. The cardinality of $\mathcal S_1$ is $\frac{q(q-1)}{2}$ when $-1 \notin {\mathbb F_q^*}^2$ and $\frac{q(q-3)}{2}$ when $-1 \in  {\mathbb F_q^*}^2$. The cardinality of $\mathcal S_2$ in both the cases is $q(q-1)\left(\frac{q-1}{2}\right) + 1$. Since $\mathcal S_1$ and $\mathcal S_2$ are disjoint, we get $s(2,q)$ by adding the two. 

Now we count conjugacy classes that are squares. The two central classes $\pm 1$ are square. The element $-1$ is a square of any conjugate of $\n$. The split regular semisimple classes are of the form $\h(a)$ with $a\neq \pm 1$. We know that, $\h(a)$ has a square root if and only if $X^2-a$ has a root in $\mathbb F_q$. Thus, the classes which are square are of the form $\h(a^2)$ with $a^2\not\in \{0, \pm 1\}$. These are $\lfloor\frac{q-3}{4}\rfloor$ classes out of total $\frac{q-3}{2}$ such classes. The non-semisimple classes are the $4$ conjugacy classes represented by $\x_{12}(1)$, $\h(-1) \x_{12}(-1)$, $\x_{12}(\epsilon)$, $\h(-1) \x_{12}(-\epsilon)$ where $\epsilon$ is a fixed non-square in $\mathbb F_q$. However, only $\x_{12}(1)$, $\x_{12}(\epsilon)$ have square roots (note that $q$ is odd). The anisotropic regular semisimple conjugacy classes are of the form $\n(-1) \x_{12}(\delta)$ with $\delta^2-4$ a non-square. Again from Proposition~\ref{BnB-n-th-root-BnB}, square root of class  $\n(-1) \x_{
12}(\delta)$ exists if and only if $X^2=2+\delta$ has a solution in  $\mathbb F_q$. These are $\lfloor\frac{q-1}{4}\rfloor$ classes out of total $\frac{q-1}{2}$ such classes. Adding all of these, we get the total number of conjugacy classes which are square $= 2+\lfloor\frac{q-3}{4}\rfloor + 2 + \lfloor\frac{q-1}{4}\rfloor = \frac{q+5}{2}$.
\end{proof}

The main hindrance in counting higher power is to count the anisotropic regular semisimple classes. Let us look at it from slightly different perspective. Let $\xi\in \mathbb F_{q^2}^*$, then left multiplication $l_{\xi}$ defines a groups homomorphism $l\colon \mathbb F_{q^2}^* \rightarrow GL_2(\mathbb F_q)$. The subgroup $\mathbb F_{q^2}^1=\{x \mid N(x)=x^{1+q}=1\}$ is of order $q+1$ and is the kernel of the norm map $N \colon \mathbb F_{q^2}^* \rightarrow \mathbb F_q^*$ given by $x\mapsto x^{1+q}$. Further, the elements $\xi\in \mathbb F_{q^2}^1$ under the map $l$ correspond to the elements in ${\rm SL}_2(\mathbb F_q)$. Note that $\mathbb F_{q^2}^1 \bigcap \mathbb F_q =\{\pm 1\}$. Under the homomorphism $l$, the elements of $\mathbb F_{q^2}^1 \backslash \mathbb F_q = \mathbb F_{q^2}^1 \backslash \{\pm 1\}$ correspond to the anisotropic regular semisimple conjugacy classes of ${\rm SL}_2(\mathbb F_q)$. Notice that two elements correspond to the same conjugacy class and hence the number of conjugacy classes is $\frac{
q+1-2}{2}= \frac{q-1}{2}$. Let $g=l_{\xi}$ be a representative of an anisotropic regular semisimple class and we wish to solve the equation $X^n=l_{\xi}$. First observe that, if a solution $x$ to this equation exists in ${\rm SL}_2(\mathbb F_q)$, it must be in $B\n B$ and correspond to an anisotropic regular semisimple class, say represented by $\zeta$, that is, we would have $yl_{\zeta}^ny^{-1}=l_{\xi}$. This amounts to finding solution of $X^n=\xi$ in $\mathbb F_{q^2}^1 \backslash \{\pm 1\}$. Hence we have,
\begin{lemma}\label{ars-power}
With notation as above and $q$ odd, the number of anisotropic regular semisimple classes in ${\rm SL}_2(\mathbb F_q)$ which are $n^{\rm th}$ power is as follows 
$$
\begin{cases}
\frac{q-1}{2}, & \text{if } (n, q+1)=1  \\
\frac{1}{2}\left(\frac{q+1}{d} - 1\right),  &  \text{if } (n, q+1)>1 \ and -1\notin ({\mathbb F_{q^2}^1})^n  \\
\frac{1}{2}\left(\frac{q+1}{d} - 2\right), &  \text{if } (n, q+1)>1 \ and -1\in ({\mathbb F_{q^2}^1})^n. 
\end{cases}
$$
where $d$ is the size of kernel of the map $\mathbb F_{q^2}^1 \rightarrow \mathbb F_{q^2}^1$ given by $x\mapsto x^n$.
\end{lemma}
\begin{corollary}\label{ars-prime-counting}
When $n \geq 3$, a prime, the number of anisotropic regular semisimple conjugacy classes in ${\rm SL}_2(\mathbb F_q)$ which are $n^{\rm th}$ power is  
$$
\begin{cases}
\frac{q-1}{2}, & \text{if } n\nmid q+1  \\
\frac{q+1}{2n} - 1, & \text{if } n\mid q+1. 
\end{cases}
$$
\end{corollary}
\begin{proof}
Since $n$ is an odd prime, $\pm 1$ both are $n^{\rm th}$ power. Rest follows from the Lemma above. 
\end{proof}
\begin{corollary}\label{ars-4-counting}
When $n=4$ the total number of anisotropic regular semisimple classes in ${\rm SL}_2(\mathbb F_q)$ which are fourth power is, 
$$
\begin{cases}
\frac{q-1}{4}, & \text{if } q \equiv 1 \imod 4  \\
\frac{q-3}{8}, & \text{if } q\equiv 3 \imod 8  \\
\frac{q-7}{8}, & \text{if } q\equiv 7 \imod 8. 
\end{cases}
$$
\end{corollary}
\begin{proof} Since $q$ is odd, $4$ is never co-prime to $q+1$.
We note that the size of kernel of the map $\mu_4\colon \mathbb F_{q^2}^1 \rightarrow \mathbb F_{q^2}^1$ given by $x\mapsto x^4$ is $4^{\rm th}$ root of unity, and hence, 
$$
d= \begin{cases} 2 & \text{if}\ q\equiv 1 \imod 4, \\
4  &  \text{if}\ q\equiv 3 \imod 4.
\end{cases}$$
Now we need to determine when $-1$ is in the image of $\mu_4$. We know that $-1$ is a fourth power in the field $\mathbb F_{q^2}$ if and only if $q^2 \equiv 1 \imod 8$, which is if and only if $q\equiv \pm 1 \imod 8$. In the case, $q\equiv 1\imod 8$, the elements $-1$ has fourth root in the base field $\mathbb F_q$ itself and hence it is not in the image of $\mathbb F_{q^2}^1$ under norm map (only $\pm 1$ in $\mathbb F_q$ have this property). Thus, if $-1$ has fourth root in $\mathbb F_{q^2}$ and the root is a norm $1$ element, happens, if and only if when $q\equiv 7 \imod 8$. We get the counting using Lemma above.
\end{proof}

\begin{proof}[\bf{Proof of Theorem~\ref{count-prime-powers}(2)}]
Now, we explicitly count the number of conjugacy classes as well as elements that are fourth powers in ${\rm SL}_2(\mathbb F_q)$. We analyse each conjugacy class one-by-one. The element $1$ is a fourth power of itself. However the fourth root of $-1$ need not exist in ${\rm SL}_2(\mathbb F_q)$ always, this happens if and only if  $q\equiv \pm 1 \imod 8$. Thus both of these classes are fourth power if and only if  $q\equiv \pm 1 \imod 8$. The split regular semisimple classes are represented by elements $\h(a)$ with $a\neq \pm 1$. Clearly $\h(a)$ has a fourth root if and only if $X^4-a=0$ has a solution in $\mathbb F_q$.  Hence the total number is $\frac{q-1}{2d}-1$ if $q\equiv 1 \imod 8$, else it is $\frac{q-1}{2d}-\frac{1}{2}$, where $d=(q-1,4)$. Which we tabulate below: 
\[
\begin{tabular}{|c|c|} \hline
$q \imod 8$ &  total number \\ \hline 
1 &  $\frac{q-9}{8}$ \\ \hline
3 &  $\frac{q-3}{4}$ \\ \hline
5 &  $\frac{q-5}{8}$ \\ \hline 
7 &  $\frac{q-3}{4}$. \\ \hline
\end{tabular}
\]
The non-semisimple classes are $4$ conjugacy classes of this kind given by $\x_{12}(1)$, $\h(-1)\x_{12}(-1)$, $\x_{12}(\epsilon)$, $\h(-1) \x_{12}(-\epsilon)$ where $\epsilon$ is a fixed non-square in $\mathbb F_q$. Representatives of all these classes are of the form $\pm\x_{12}(\psi)$ with $\psi\neq 0$. Out of these only $\x_{12}(1)$ and $\x_{12}(\epsilon)$ are fourth power. The anisotropic regular semisimple conjugacy classes are $\n(-1)\x_{12}(\delta)$ such that $X^2-\delta X+1$ is irreducible over $\mathbb F_q$. Using Corollary~\ref{ars-4-counting} we get the total number of these classes which are fourth power as in the following table:
\[
\begin{tabular}{|c|c|} \hline
$q \imod 8$ & total number \\ \hline 
1 &  $\frac{q-1}{4}$ \\ \hline
3 &  $\frac{q-3}{8}$ \\ \hline
5 &  $\frac{q-1}{4}$ \\ \hline 
7 &  $\frac{q-7}{8}$. \\  \hline 
\end{tabular}
\]
Thus, we get the total number of conjugacy classes which are fourth power and get the required result. 
Now we count the number of elements. We do this counting case-by-case as follows.
\begin{enumerate}
\item When $q\equiv 1 \imod 8$ the total number of elements which are fourth power is
$$2.1 + \frac{q-9}{8}.q(q+1) + 2. \frac{q^2-1}{2} + \frac{q-1}{4}.q(q-1) = \frac{3}{8}q^3 - \frac{1}{2}q^2 -\frac{7}{8}q +1.$$
\item When $q\equiv 3 \imod 8$ the total number of elements which are fourth power is
$$1.1 + \frac{q-3}{4}.q(q+1) + 2. \frac{q^2-1}{2} + \frac{q-3}{8}.q(q-1) = \frac{3}{8}(q^3-q).$$
\item When $q\equiv 5 \imod 8$ the total number of elements which are fourth power is
$$1.1 + \frac{q-5}{8}.q(q+1) + 2. \frac{q^2-1}{2} + \frac{q-1}{4}.q(q-1) = \frac{3}{8}(q^3-q).$$
\item When $q\equiv 7 \imod 8$ the total number of elements which are fourth power is
$$2.1 + \frac{q-3}{4}.q(q+1) + 2. \frac{q^2-1}{2} + \frac{q-7}{8}.q(q-1) = \frac{3}{8}q^3 - \frac{1}{2}q^2 +\frac{1}{8}q +1.$$
\end{enumerate}
\end{proof}

\begin{proof}[\bf{Proof of Theorem~\ref{count-prime-powers}(3)}]
The proof is similar to the earlier cases. Since $n$ divides the order of group, $n$ divides exactly one of  $q-1, q, q+1$. 
\begin{enumerate}
\item When $n\mid q$, that also means $n\nmid (q-1)(q+1)$ hence,
 $$c(n,q)= 2 + \frac{q-3}{2} + 0 + \frac{q-1}{2} = q$$
and 
$$s(n,q)= 2.1 + \frac{q-3}{2}.q(q+1) + 0.\frac{q^2-1}{2} + \frac{q-1}{2}.q(q-1) = (q-2)(q^2-1).$$
\item When $n \mid (q-1)$ that also implies $n \nmid q(q+1)$ hence,
$$c(n,q)= 2 + \left(\frac{q-1}{2n} - 1\right) + 4 + \frac{q-1}{2} = \frac{(n+1)(q-1)}{2n} + 5$$
and 
$$s(n,q)= 2.1 + \left(\frac{q-1}{2n} - 1\right).q(q+1) + 4.\frac{q^2-1}{2} + \frac{q-1}{2}.q(q-1) = \frac{(n+1)(q^3-q)}{2n}.$$
\item Similarly, when $n \mid (q+1) $ gives $n\nmid (q-1)q$ thus,
$$c(n,q)= 2 + \frac{q-3}{2} + 4 + \left(\frac{q+1}{2n} - 1\right) =  \frac{(n+1)(q-3)+4}{2n} + 5$$
and 
$$s(n,q)= 2.1 + \frac{q-3}{2}.q(q+1) + 4.\frac{q^2-1}{2} + \left(\frac{q+1}{2n} - 1\right).q(q-1) = \frac{(n+1)(q^3-q)}{2n}.$$
\end{enumerate}
\end{proof}
As a consequence to our counting above, we give an alternate proof to a well known Waring type result (see~\cite{lost, lu, gm} for more general results) for the groups ${\rm SL}_2(\mathbb F_q)$.
\begin{corollary}\label{two-n-powers}
Let $n > 2$ be a prime and $q$ odd. Then the word map $X_1^nX_2^n$ is surjective on ${\rm SL}_2(\mathbb F_q)$ except when $n=3=q$. 
\end{corollary}
\begin{proof}
If $n$ does not divide the order of the group the word map $X^n$ itself is surjective. Thus we need to look at the case when $n$ divides the order of the group.
If $n\nmid q$ then from Theorem~\ref{count-prime-powers} it is clear that the number of $n^{\rm th}$ powers is $\frac{n+1}{2n} > \frac{1}{2}$. Hence, the product of two such elements will cover whole of the group. 

Now, we are left with the case $n\mid q$. The proportion of elements which are $n^{\rm th}$ powers in ${\rm SL}_2(\mathbb F_q)$ is $\frac{q-2}{q}$ which is $> \frac{1}{2}$ if $q\geq 5$, and we are done. This leaves us with the case when $q=3$ and $n=3$. In this case non-semisimple classes are not cubes and neither product of cubes, as discussed in the following paragraph. Thus the proof concludes.
\end{proof}
We now discuss the exception case $n = 3 = q$ of the above theorem. In the group ${\rm SL}_2(\mathbb F_3)$, the word map $X_1^3 X_2^3$ is not surjective. The group ${\rm SL}_2(\mathbb F_3)$ has $7$ conjugacy classes. Of these, $2$ conjugacy classes are central, $4$ are non-semisimple and $1$ is anisotropic regular semisimple corresponding to the irreducible polynomial $X^2 + 1$ over $\mathbb F_3$. The non-semisimple conjugacy classes are not cube and others are cube. Hence the number of conjugacy classes which are cube is $3$ and the number of elements in ${\rm SL}_2(\mathbb F_3)$ which are cube is  $2.1 + 1. 3(3 - 1) = 8$. These are the following elements.
{\small \begin{center}
$$\begin{pmatrix} 1 & 0 \\ 0 & 1 \end{pmatrix}, 
\begin{pmatrix} -1 & 0 \\ 0 & -1 \end{pmatrix}, 
\begin{pmatrix} 0 & 1 \\ -1 & 0 \end{pmatrix}, 
\begin{pmatrix} 1 & 1 \\ 1 & -1 \end{pmatrix}, 
\begin{pmatrix} 1 & -1 \\ -1 & -1 \end{pmatrix}, 
\begin{pmatrix} 0 & -1 \\ 1 & 0 \end{pmatrix}, 
\begin{pmatrix} -1 & 1 \\ 1 & 1 \end{pmatrix}, 
\begin{pmatrix} -1 & -1 \\ -1 & 1 \end{pmatrix}.
$$
\end{center}}
\noindent These $8$ elements form a subgroup of ${\rm SL}_2(\mathbb F_3)$ isomorphic the quaternion group. Thus we conclude that in ${\rm SL}_2(\mathbb F_3)$ product of cubes is again a cube. Therefore non-cubes in this group, which are $16$ in count, can never be products of any number of cubes. In particular, the word map $X_1^3 X_2^3 \dots X_l^3$ on ${\rm SL}_2(\mathbb F_3)$ is not 
surjective for any $l\geq 1$.

\section{Product of powers in groups of type $A_1$}
As an application of our methods, we discuss products of powers in the forms of ${\rm SL}_2$ over base field $k$ when ${\rm char}(k)\neq 2$. 
We recall that groups $G$ of the type $A_1$ are the following.
\begin{enumerate}
\item[I.] (Split case) $G = {\rm SL}_2(k)$.
\item[II.] (Anisotropic case) $G = {\rm SL}_1(Q)$, where $Q$ is a quaternion division algebra over $k$.
\end{enumerate}

When $k = \mathbb F_q$, results are known for a larger class of groups (see~\cite{lost, lu, gm}). 
The results for anisotropic case, and for split case with arbitrary $k$ are new.


\begin{theorem}\label{waring-type-main-result}
Let $k$ be a field with ${\rm char}(k) \neq 2$. Let $G$ be a group of type $A_1$ over $k$.
\begin{enumerate}
\item If $G$ is split then the word map $X_1^n X_2^n$ is surjective on $G$ if $n = 2$, or $n = 3$ and ${\rm char}(k) \neq 3$.
\item If $G$ is anisotropic and $2 \in (k^*)^2$ then the word map $X_1^2 X_2^2$ is surjective on $G$ if and only if $-1$ is a square in $G$.
\item If $G$ is anisotropic and $n \geq 3$ is odd, then the word map $X_1^n X_2^n$ is surjective if $-1$ is a square in $G$.
\end{enumerate}
\end{theorem}

We prove this theorem in the rest of this section.

\begin{proof}[\bf{Proof of Theorem~\ref{waring-type-main-result}(1)}]
We use the description of squares in ${\rm SL}_2(k)$ obtained in Lemma~\ref{squares-in-sl2}, along with the notation $\mathcal S_1$ and $\mathcal S_2$
therein. We first prove that all elements of $B\n B$ are products of two squares. For this,  we take $x_1^2\in \mathcal S_1$ and $x_2^2\in B\n B$. The product is $x_1^2x_2^2 =  \h(a_1^2)\x_{12}(t_1).\x_{12}(t_2) \n(a_2)  \x_{12}(-t_2 + a_2(s_2^2 - 2))  = \x_{12}(a_1^{4}(t_1+t_2)) \n(a_1^2a_2) \x_{12}(-t_2 + a_2(s_2^2-2))$.
By taking $a_1=1$, we see that we get all elements of $B\n B$. Now, to get all elements of $B$ we take products from $\mathcal S_2$ as follows, $x_1^2x_2^2 =  \x_{12}(t_1)\n(a_1) \x_{12}(-t_1 +a_1(s_1^2-2)). \x_{12}(t_1 - a_1(s_1^2-2)) \n(a_2) \x_{12}(-t_1 +a_1(s_1^2-2) +a_2(s_2^2-2)) = \h(-a_1a_2^{-1}) \x_{12}(-t_1+ a_1(s_1^2-2)+a_2(s_2^2-2) + a_1^{-2}a_2^{2}t_1)$. This covers all elements of $B$ except when $a_1=\pm 1$. Thus the elements which are left out so far are $\pm \x_{12}(t)$. These are obtained by multiplying elements of $\mathcal S_1$ together with $-1$ from $\mathcal S_2$. This completes the
case of product of two squares.

For the case of product of two cubes, we note that the elements of the form $\pm \x_{12}(\psi)$ are always cube. Other elements $\h(a)\x_{12}(s)$ of $B$ are cube if $a\in (k^*)^3$. More importantly, $\x_{12}(\tau)\n(\alpha)\x_{12}(\psi)$ with $\tau+\psi=2\alpha$ are certainly cube. Let us consider product of cubes of the form $\x_{12}(\tau)\n(\alpha)\x_{12}(2\alpha-\tau)$ which are in $B\n B$. Take the product when last term of $x_1^3$ is same as inverse of the first term of $x_2^3$. We get,
\begin{eqnarray*}
x_1^3x_2^3 &=&  \x_{12}(\tau_1)\n(\alpha_1)\x_{12}(2\alpha_1-\tau_1). \x_{12}(-2\alpha_1+\tau_1)\n(\alpha_2) \x_{12}( 2\alpha_2+ 2\alpha_1-\tau_1)\\
&=&  \x_{12}(\tau_1)\h\left(-\frac{\alpha_1}{\alpha_2}\right) \x_{12}( 2\alpha_2+ 2\alpha_1-\tau_1)\\
&=& \h\left(-\frac{\alpha_1}{\alpha_2}\right) \x_{12}\left(\left (\frac{\alpha_2^2}{\alpha_1^2} -1\right) \tau_1 + 2(\alpha_1+ \alpha_2)\right).
\end{eqnarray*}
Combined with the fact that $\pm\x_{12}(\psi)$ are already cube, all elements of $B$ are a product of two cubes. Now, let us compute the product when $x_1^3=\x_{12}(\psi)$ and $x_2^3$ is in $B\n B$ of the above kind. We get, 
\begin{eqnarray*}
x_1^3x_2^3 &=&   \x_{12}(\tau_1). \x_{12}(\tau_2)\n(\alpha_2)\x_{12}(2\alpha_2-\tau_2)\\
&=& \x_{12}(\tau_1+ \tau_2)\n(\alpha_2)\x_{12}(2\alpha_2-\tau_2).
\end{eqnarray*}
This shows that all elements of $B\n B$ are also a product of two cubes.
This completes the proof
of \ref{waring-type-main-result}(1). 
\end{proof}



To prove Theorem \ref{waring-type-main-result}(2) and \ref{waring-type-main-result}(3)  we set up some lemmas. 
Let $k$ be a field of characteristic $\neq 2$ and  $Q = \left(\frac{a,\lambda}{k}\right)$ be a quaternion division algebra over $k$. We recall that every element $\psi$ of $Q$ satisfies the quadratic equation $\psi^2-Tr(\psi) \psi + N(\psi)=0$.
\begin{lemma}\label{Q-square}
Let $(\alpha, \beta) \in {\rm SL}_1(Q)$ with $\beta\neq 0$. The, $X^2=(\alpha,\beta)$ has a solution in ${\rm SL}_1(Q)$ if and only if $Tr(\alpha) + 2\in (k^*)^2$.
\end{lemma}
\begin{proof}
Let us first assume $Tr(\alpha) + 2\in (k^*)^2$. Let $(x,y)\in Q$ be a solution of $X^2=(\alpha,\beta)$. Then, $(x,y)^2= Tr(x)(x,y) - 1$ gives the equations $Tr(x)x-1 = \alpha$ and $Tr(x)y=\beta$. For the simplicity of further calculation, let us write $x=l+m\zeta$ and $\alpha=\alpha_1 + \alpha_2\zeta$ where $\alpha_1, \alpha_2, l, m\in k$. In case $Tr(\alpha)=2 \alpha_1 \neq -2$, i.e., $\alpha_1\neq -1$, the equation $2l^2-1=\alpha_1$ give $l\neq 0$ and $y = \frac{\beta}{2l}$. Hence, $l^2= \frac{Tr(\alpha)+2}{4}$ and $x= \frac{\alpha +1}{2l}$ gives the solution. In the case $\alpha_1=-1$, we get $l=0$, i.e., $Tr(x)=0$, thus $\alpha=-1$ and $\beta=0$ which is not the case.

Now, suppose $X^2=(\alpha, \beta)$ has a solution, say $(x,y)$. Then, the equation $Tr(x)x=\alpha+1$, after taking trace, gives that  $Tr(x)^2=Tr(\alpha)+2$. Clearly $Tr(x)\neq 0$ else $\beta =0$. 
\end{proof}
\begin{lemma}
Let $\phi=(\alpha, \beta) \in {\rm SL}_1(Q)$ with $Tr(\alpha)=0$ and $\beta\neq 0$. Then, for any $n$ odd, the equation $X^n=\phi$ has a solution in ${\rm SL}_1(Q)$.
\end{lemma}
\begin{proof}
We have, $\phi^2=-1$. Hence, $\phi^n=\phi$ if $n\equiv 1 \imod 4$ and $\phi^n=-\phi$ when $n\equiv 3\imod 4$. Hence, all such $\phi\in Q$ have $n^{\rm th}$ root which is either $\phi$ or $-\phi$.
\end{proof}

\begin{proof}[{\bf Proof of \ref{waring-type-main-result}(2) and \ref{waring-type-main-result}(3)}]
We prove (1) first. Let $-1$ be a square in ${\rm SL}_1(Q)$. Consider the set $\mathcal H = \{(\alpha, \beta) \mid N(\alpha, \beta)=1, Tr(\alpha)=0\}$.  The set $\mathcal H$ consists of square-root of $-1$ in ${\rm SL}_1(Q)$ and, is non-empty, further it is contained in squares (Lemma~\ref{Q-square}). We consider $(\alpha_1,\beta_1)$ and $(\alpha_2, \beta_2)$ in $\mathcal H$ and,
$$(\alpha_1,\beta_1)(\alpha_2, \beta_2)= (\alpha_1\alpha_2+\lambda \beta_1\overline \beta_2, \alpha_1\beta_2 +\overline \alpha_2\beta_1).$$
Thus, when $\alpha_1\beta_2 +\overline \alpha_2\beta_1\neq 0$, say $\gamma$, then the product becomes $(\frac{\gamma\alpha_2-\beta_1}{\beta_2},\gamma)$. This covers all elements $(x,y)\in {\rm SL}_1(Q)$ with $y\neq 0$. Now we need to produce elements of kind $(x,0)$ as a product of two squares. We note that when $\alpha_1\beta_2 +\overline \alpha_2\beta_1=0$, $(\alpha_1,\beta_1)(\alpha_2, \beta_2)= \left (-\frac{\overline\beta_2}{\overline\beta_1}, 0\right ) = \left (-\frac{\alpha_2}{\overline\alpha_1}, 0\right )$. Thus it reduces to prove that every element of $K$ is a product of two elements from the set $S=\{\alpha\in K \mid Tr(\alpha)+2\in (k^*)^2\}$ which is easy to verify. 

Now to prove converse, we begin with $-1=(\alpha_1,\beta_1)(\alpha_2, \beta_2)$ where $(\alpha_1,\beta_1)=(x_1,y_1)^2=(x_1^2+\lambda N(y_1), Tr(x_1)y_1)$ and $(\alpha_2, \beta_2) = (x_2, y_2)^2$. We may assume $Tr(x_1)$ and $Tr(x_2)$ both are not $0$ else we get $-1$ as a square. If $y_1=0$ then $-1= x_1^2 (x_2^2+\lambda N(y_2), Tr(x_2)y_2)$ hence $Tr(x_2)y_2=0$. If $y_2=0$ we get $-1= x_1^2x_2^2$ and if $Tr(x_2)=0$ we would have $(x_2,y_2)^2=-1$. Similarly we can prove this when $y_2=0$, thus we may assume $\beta_1$ and $\beta_2$ both are non-zero. Now we use ${\rm SL}_2(K)$ notation and we have, 
$$-1 = \x_{12}\left(\frac{\alpha_1}{\bar \beta_1}\right)\n\left(\frac{-1}{\bar \beta_1}\right) \x_{12}\left(\frac{\bar \alpha_1}{\bar \beta_1}\right) \x_{12}\left(\frac{\alpha_2}{\bar \beta_2}\right)\n\left(\frac{-1}{\bar \beta_2}\right) \x_{12}\left(\frac{\bar \alpha_2}{\bar \beta_2}\right)$$
in ${\rm SL}_2(K)$. From Proposition~\ref{coset-mult}  (3), we must have $\frac{\bar \alpha_1}{\bar \beta_1} = -\frac{\alpha_2}{\bar \beta_2} $. Write $\frac{\bar \alpha_1}{\alpha_2} = -\frac{\bar \beta_1}{\bar \beta_2} =\theta$ and get $(\alpha_1, \beta_1) = \bar \theta (\alpha_2, - \beta_2)$. This gives $\theta=-1$ and thus equation becomes $-1= (-\bar \alpha_2, \beta_2) (\alpha_2, \beta_2)$. Write $\alpha_2=r+\zeta s$, and we get $Tr(-\bar \alpha_2)+2= -2r+2 \in (k^*)^2$ and $Tr(\alpha_2)+2 = 2r+2\in (k^*)^2$. Thus $1-r^2\in (k^*)^2$. Now $1=N(\alpha_2, \beta_2)=r^2-as^2-\lambda N(\beta_2)$ implies that the quadratic form $<1, a, \lambda, - a\lambda>$ is isotropic which is equivalent to $-1$ being a square (see~\cite{la} Chapter III Exercise 5). In fact, it gives $\left(\frac{s\zeta}{r'}, \frac{\beta_2}{r'}\right)$ where $1-r^2=r'^2$, of which square is $-1$.

For the proof of (2), we consider the set $\mathcal H=\{(\alpha, \beta) \in {\rm SL}_1(Q) \mid Tr(\alpha)=0, \beta\neq 0\}$ contained in $n^{\rm th}$ powers. For the set $\mathcal H$ to be non-empty we require $-1$ to be a square. Rest of the proof is similar to as above.
\end{proof}
\noindent We end this section with some examples.
\begin{example}
Let $k=\mathbb Q$ and $g=\begin{pmatrix} & 1 \\ -1 & -a\end{pmatrix} \in {\rm SL}_2(\mathbb Q)$ where $a\neq 0$. Clearly, $g=\n \x_{12}(a)$. From Lemma~\ref{squares-in-sl2} it follows that $x^2=g$ has a solution in ${\rm SL}_2(k)$ if, and only if, $8-2a^2 \in \mathbb Q^2$. Thus, if $a\geq 5$ then $g$ does not have a square root in ${\rm SL}_2(\mathbb Q)$. Similarly, we can produce elements which are not $n^{\rm th}$ power in ${\rm SL}_2(\mathbb F_q)$ using the discussion preceding Lemma~\ref{ars-power}.
\end{example}

\begin{example}
Consider $k=\mathbb R$ and $Q=\mathbb H$, the Hamilton's quaternion. Consider an element $(\alpha, \beta)\in {\rm SL}_1(\mathbb H)$ with $\beta\neq 0$. Any such element with $Tr(\alpha) \leq -3$ is not a square using Lemma~\ref{Q-square}. 
\end{example}


\end{document}